\documentclass[runningheads,a4paper]{llncs}

\usepackage{cite}
\usepackage[cmex10]{amsmath}
\usepackage{amssymb}
\usepackage{url}
\usepackage[all]{xy}
\usepackage{latexsym}
\usepackage{xspace}

\input diagxy

\newcommand{\ran}[0]  {
  \mathit{Ran}
}
\newcommand{\lan}[0]  {
  \mathit{Lan}
}
\newcommand{\lif}[0]  {
  \mathit{Lift}
}
\newcommand{\disc}[0]  {
  \mathit{Disc}
}
\newcommand{\comma}[2]  {%
{#1\!\!\downarrow\!\!{#2}}
}

\newcommand{\catl}[1]  {
  \mathbb{#1}
}
\newcommand{\catw}[1]  {
  \mathbf{#1}
}
\newcommand{\onecat}[0]  {
  \catw{Cat}
}
\newcommand{\cat}[0]  {
  \catw{cat}
}
\newcommand{\word}[1]  {
  \mathit{#1}
}
\newcommand{\op}[1]  {
{{#1}^\word{op}}
}
\newcommand{\mor}[3]  {
  #1 \colon #2 \rightarrow #3
}
\newcommand{\dist}[3]  {
  #1 \colon #2 \nrightarrow #3
}
\newcommand{\tuple}[1]  {
	\langle #1 \rangle
}
\newcommand{\id}[1]  {
  \word{id}_{#1}
}
\newcommand{\yon}[4]  {
  #1 \colon #2 \triangleright \langle #3, #4 \rangle
}
\newcommand{\ph}[1] {%
	{{-}_{#1}}
}
\newbox\anglebox
\setbox\anglebox=\hbox{\xy \POS(75,0)\ar@{-} (0,0) \ar@{-} (75,75)\endxy}

\newbox\aanglebox
\setbox\aanglebox=\hbox{\xy \POS(0,0)\ar@{-} (0,75) \ar@{-} (75,0)\endxy}

\begin{document}

\title{Logical systems II:\\Free semantics}

\author{Michal R. Przybylek}
\institute{Faculty of Mathematics, Informatics and Mechanics\\
University of Warsaw\\
Poland}
\maketitle

\begin{abstract}
This paper is a sequel to \cite{log1}. It provides a general 2-categorical setting for extensional calculi and shows how intensional and extensional calculi can be related in logical systems. We define Yoneda triangles as relativisations of internal adjunctions, and use them to characterise universes that admit a notion of convolution. We show that such universes induce semantics for lambda calculi. We prove that a construction analogical to enriched Day convolution works for categories internal to a locally cartesian closed category with finite colimits.
\end{abstract}

\section{Introduction}
\label{c:free:semantics}
A well-known Mac Lane's slogan says ``adjunctions arise everywhere''. One may find adjunctions in variety of concepts from theoretical computer science: definitions of Galois correspondence between syntax and semantics together with Dedekind-MacNeille completion as the fixed-point of the adjunction, power objects, function spaces (in the form of lambda abstraction), logical connectives and quantifications, to classical mathematics: free structures, definitions of tensor products, distributions, and many more. In some cases, however, the definition of an adjunction is too restrictive. The following example is taken from \cite{log1}.
\begin{example}[Topological spaces]\label{e:top:closed}
Although category of topological spaces is not cartesian closed, very many interesting topological spaces are exponentiable. In fact for a topological space $A$ there exists right adjoint to $\mor{- \times A}{\catw{Top}}{\catw{Top}}$ if and only if $A$ is a core-compact space \cite{isbellspace}, which means that the underlying locale of its open sets is continuous. One then may think that a restriction to the subcategory of topological spaces consisting of core-compact spaces could work. However, this again is not the case, because an exponent of two core-compact spaces need not be core-compact.
\end{example}
The example shows that a category may not be closed in itself, but in a bigger embedding category. Actually, the above situation is quite simple, because objects from the subcategory were exponentiable with respect to \emph{all} objects in the embedding category. Here is a less trivial example.
\begin{example}[Partially recursive functions]\label{e:prf}
Let us consider a category consisting of two objects --- the set of natural numbers $\mathbb{N}$, together with partially recursive functions, and a singleton $1$, with all singleton maps $1 \rightarrow \mathbb{N}$. This category is cartesian with binary products given by any effective pairing. It is not, however, cartesian closed --- the evaluation cannot be partially recursive --- was it, one could test for equality of two partially recursive functions by checking the equality of their corresponding natural numbers\footnote{In other words --- if the coding is effective then it has to be ambiguous.}. On the other hand, it is ``closed'' in the category of $\Pi_2^0$ functions in the sense of arithmetic hierarchy. Similarly, the category of $\Pi_2^0$ functions is not cartesian closed in itself. In fact, a simple diagonalisation argument shows that none of $\Sigma_n^m, \Pi_n^m, \Delta_n^m$ is cartesian closed.
\end{example}
Such situations frequently occur when a construction over an object is of a poorer quality than the original object. Here is our driving example
\begin{example}[Russell paradox]\label{e:russelk:paradox}
In a ZFC set theory\footnote{The result generalises to any higher-order type theory \cite{forti}.} there can be no set $\mathcal{U}$ universal for all sets --- i.e.~there is no set $\mathcal{U}$ such that every set is isomorphic to exactly one element of $\mathcal{U}$. 
However, there exists a (necessarily proper\footnote{Otherwise, by the axiom of union we could form $A_0 = \bigcup_k \bigcup \mathcal{U}_k$, and $A = P(A_0)$ would not be classified by any $\mathcal{U}_k$.}) family of sets $\mathcal{U}_0 \subseteq \mathcal{U}_1 \subseteq \mathcal{U}_2 \subseteq \cdots$ that is collectively universal, which means that for every set $A$ there exists $\mathcal{U}_k$ and exactly one $X \in \mathcal{U}_k$ with $A \approx X$.
\end{example}
We should think of universes as 2-dimensional analogue of the internal truth-values object $\Omega$ in a topos --- just like $\Omega$ classifies \emph{internal logic} of a category, a universal object tries to classify the \emph{external logic}. The attempt to classify the full external logic is, however, futile, as stated in the above example. Therefore, we have to focus on a classification of some parts of the external logic.

In this paper we set forth categorical foundations for ``2-powers'', which shall generalise partial classifiers of external logics for objects in a 2-category. We show how internal logical systems in any 2-category with 2-powers carry free semantics on objects. We propose a notion of a Yoneda (bi)triangle and use it to generalise to the 2-categorical setting the famous construction of convolution introduced by Brain Day in his PhD dissertation \cite{day} for enriched categories. As a complementary result we prove the convolution theorem for internal categories.

\section{Categorical 2-powers} \label{s:yoneda:triangle}
To better understand our definition of ``2-powers'', let us first recall how one may define ordinary powers. With every regular category\footnote{A category is called regular (Chapter 2, Volume 2 of \cite{borceux}, Chapter A1.3 of \cite{elephant}, Chapter 4, Section 4 of \cite{jacobs},  (Chapter 1.5 of \cite{scedrov}) if it has finite limits, regular epimorphisms are stable under pullbacks, and every morphism factors as a regular epimorphism followed by a monomorphism.} $\catl{C}$ there is associated a 2-poset (i.e.~poset-enriched\footnote{Poset-enriched category, also called 2-poset, is a category enriched in the category of posets.}) of its internal relations $\catw{Rel}(\catl{C})$ together with a bijective-on-objects functor $\mor{J}{\catl{C}}{\catw{Rel}(\catl{C})}$. Furthermore, the right adjoint of $J$, if it exists, $\mor{P}{\catl{C}}{\catw{Rel}(\catl{C})}$ induces the natural isomorphism: 
\begin{center}
\begin{tabular}{c}
$\hom_{\catw{Rel}(\catl{C})}(J(A), B)$\\
\hline\hline
$\hom_\catl{C}(A, P(B))$\\
\end{tabular}
\end{center}
If additionally $\catl{C}$ has a terminal object $1$ then, recalling the definition of an internal relation, one gets:
\begin{center}
\begin{tabular}{c}
$\mathit{sub}(\catl{C})(A)$\\
\hline\hline
$\hom_{\catw{Rel}(\catl{C})}(J(A), 1)$\\
\hline\hline
$\hom_\catl{C}(A, P(1))$\\
\end{tabular}
\end{center}
which makes $\catl{C}$ a topos with power functor $P$ and the subobject classifier $\Omega = P(1)$. All of the above may be abstractly characterised by starting with a regular fibration $\mor{p}{\catl{E}}{\catl{C}}$ on a finitely complete category $\catl{C}$, then constructing the category of $p$-internal relations $\catw{Rel}(p)$ and a bijective-on-objects functor $\mor{J}{\catl{C}}{\catw{Rel}(p)}$. We shall recover the classical situation by taking for $p$ the usual subobject fibration. Now we would like to argue that the right notion of the category of relations over $\catl{C}$ is encapsulated by \emph{any} bijective-on-objects functor $\mor{J}{\catl{C}}{\catl{D}}$, where $\catl{D}$ is a 2-poset. Here are some intuitions. Let us recall that any such bijective-on-objects functor corresponds to a poset-enriched module monad ($\catw{Pos}$ is the category of partially ordered sets and pointwise-ordered monotonic functions):
$$\mor{\hom_\catl{D}(J(\ph{1}), J(\ph{2}))}{\catl{C}^{op} \times \catl{C}}{\catw{Pos}}$$
with unit:
$$\mor{\eta}{\hom_\catl{C}(\ph{1}, \ph{2})}{\hom_\catl{D}(J(\ph{1}), J(\ph{2}))}$$
given by the action of $J$, and multiplication:
$$\mor{\mu}{\int^{B \in \catl{C}}\hom_\catl{D}(J(\ph{1}), J(B)) \times \hom_\catl{D}(J(B), J(\ph{2}))}{\hom_\catl{D}(J(\ph{1}), J(\ph{2}))}$$
induced by the composition from $\catl{D}$. This monad gives a ``fibred span'' (i.e.~a span where one leg is a fibration and the other is an opfibration)  ${\catl{C} \overset{\pi_1}\leftarrow \int \hom_\catl{D}(J(\ph{1}), J(\ph{2})) \overset{\pi_2}\rightarrow \catl{C}}$ with a monoidal action induced by a generalised Grothendieck construction --- the total category is defined as the following coend:
$$\int \hom_\catl{D}(J(\ph{1}), J(\ph{2})) = \int^{X,Y \in \catl{C}} \catl{C}/X \times \hom_\catl{D}(J(X), J(Y)) \times Y/\catl{C}$$
where:
\begin{eqnarray*}
\mor{\catl{C}/(-)}{\catl{C}}{\onecat} \\
\mor{(-)/\catl{C}}{\op{\catl{C}}}{\onecat}  
\end{eqnarray*}
are the slice an coslice functors,  and $\pi_1, \pi_2$ are the obvious projections.
If we assume that $\catl{C}$ has a terminal object $1$, then:
$$\mor{\hom_\catl{D}(J(1), J(-))}{\catl{C}}{\catw{Pos}}$$
by Grothendieck construction corresponds to an opfibration:
$$\mor{\pi_{\hom_\catl{D}(J(1), J(-))}}{\int \hom_\catl{D}(J(1), J(-))}{\catl{C}}$$
and:
$$\mor{\hom_\catl{D}(J(-), J(1))}{\catl{C}^{op}}{\catw{Pos}}$$
corresponds to a fibration:
$$\mor{\pi_{\hom_\catl{D}(J(-), J(1))}}{\int \hom_\catl{D}(J(-), J(1))}{\catl{C}}$$
In case $\catl{D} = \catw{Rel}(p)$ these two functors are ``essentially the same'' and encode the fibration $\mor{p}{\catl{E}}{\catl{C}}$; one may check that our fibred span arises by pulling back $\mor{p}{\catl{E}}{\catl{C}}$ along the Cartesian product functor $\mor{\times}{\catl{C} \times \catl{C}}{\catl{C}}$ to obtain a bifibration $\mor{\mathit{rel}(p)}{\catw{Rel}(p)}{\catl{C} \times \catl{C}}$ and postcomposing it with two projections. For this reason, the functor $\mor{J}{\catl{C}}{\catl{D}}$ does not lose any information about the regular logic associated to $\catl{C}$. Second, we do believe that a more natural setting for relations is a fibred span than a bifibration --- this allows us to distinguish between relations $A \nrightarrow B$ from relations $B \nrightarrow A$ and generalise the construction to higher categories. For example, as suggested by Jean Benabou, the role of relations between categories should be played by distributors. For any complete and cocomplete symmetric monoidal closed category $\catl{C}$, we may define a 2-distributor:
$$\mor{\hom_{\catw{Dist}(\catl{C})}}{\catw{Cat}(\catl{C})^{op} \times \catw{Cat}(\catl{C})}{\catw{Cat}}$$
sending $\catl{C}$-enriched categories $A, B$ to the category of $\catl{C}$-enriched distributors $A \nrightarrow B$ and $\catl{C}$-enriched natural transformations. Because:
$$\hom_{\catw{Dist}(\catl{C})}(A, B) \not\approx \hom_{\catw{Dist}(\catl{C})}(B, A)$$
the 2-distributor $\hom_{\catw{Dist}(\catl{C})}$ is not induced by any (co)indexed 2-category.

\begin{example}[Allegory]\label{e:allegory}
Another way to look at these concepts is through the notion of an \emph{allegory} \cite{scedrov}. An allegory is a pair $\tuple{\catl{A}, \mor{(-)^\star}{\catl{A}}{\op{\catl{A}}}}$, where $\catl{A}$ is a poset-enriched category (2-poset), $\mor{(-)^\star}{\catl{A}}{\op{\catl{A}}}$ is an identity-on-objects duality involution, and:
\begin{itemize}
\item for each $A, B \in \catl{A}$ the poset $\hom(A, B)$ has binary conjunctions
\item for each triple of morphisms $\xy \morphism(0,0)|b|<200,0>[A`B;f] \morphism(200,0)|b|<200,0>[B`C;g] \morphism(0,0)|a|/{@{>}@/^1em/}/<400,0>[A`C;h]\endxy$ the following holds:
$$(g \circ f) \wedge h \leq g \circ (g^\star \circ h \wedge f)$$
\end{itemize}
Every allegory $\catl{A}$ induces a bijective on objects embedding $\mor{J}{\catl{C}}{\catl{A}}$ by forming a subcategory $\catl{C}$ consisting of morphisms that have right adjoints. Moreover, if $\catl{A}$ is a tabular allegory\footnote{An allegory is tabular if every morphism $h$ admits a decomposition $h = g^\star \circ f$ such that $f^\star \circ f \wedge g^\star\circ g = \mathit{id}$.}, then $\catw{Rel}(\catl{C}) \approx \catl{A}$ and $\catl{C}$ is (locally) regular (2.148 in \cite{scedrov}; moreover, the converse is true by 2.132 in the same book).
\end{example}

As mentioned earlier, a (locally small) category with finite limits has power objects iff it is regular and the induced functor $\mor{J}{\catl{C}}{\catw{Rel}(\catl{C})}$ has right adjoint. It is natural then to provide the following generalisation of a power functor. If $\mor{J}{\catl{C}}{\catl{D}}$ is a bijective on objects functor, then we say that $P(B) \in \catl{C}$ is a $J$-power of $B \in \catl{D}$ if there is a representation:
$$\hom_\catl{D}(J(-), B) \approx \hom_\catl{C}(-, P(B))$$
If a representation $P(B)$ exists for every $B \in \catl{D}$, i.e.~$J$ has the right adjoint $\mor{P}{\catl{D}}{\catl{C}}$, we say that $\catl{C}$ has $J$-powers.
\begin{example}[Topos]
Let $\catl{C}$ be a (locally small) regular category, and $\mor{J}{\catl{C}}{\catw{Rel}(\catl{C})}$ its inclusion functor into the category of relations. $\catl{C}$ is a topos iff it has $J$-powers.
\end{example}
\begin{example}[Quasitopos]
Let $\catl{C}$ be a finitely complete and cocomplete locally cartesian closed category, such that its fibration of regular subobjects\footnote{A regular subobject of $A$ is a (equivalence class of) regular monomorphism with codomain $A$. A regular monomorphism is a monomorphism that arises as an equaliser.} is regular\footnote{See Chapter 4, Section 4 of \cite{jacobs}}, and $\mor{J}{\catl{C}}{\catw{RegRel}(\catl{C})}$ its inclusion functor into the category of regular relations. $\catl{C}$ is a quasitopos iff it has $J$-powers.
\end{example}
\begin{example}[Regular fibration]
More generally, let $\mor{p}{\catl{E}}{\catl{C}}$ be a regular fibration on a finitely complete category $\catl{C}$. If $\mor{J}{\catl{C}}{\catw{Rel}(p)}$ has a right adjoint, then $\mor{p}{\catl{E}}{\catl{C}}$ has a generic object. The converse is true provided that $\catl{C}$ is cartesian closed. 
\end{example}
Still, as exposed in the introduction, such definition is too strong to embrace many interesting examples. Here is another one. Let $\cat$ be the 2-category of small categories, and $\catw{dist}$ the 2-category of distributors, with the usual bijective on objects embedding $\mor{J}{\cat}{\catw{dist}}$ defined on functors:
$$J(F) = \hom_{\catw{dist}}(\ph{1}, F(\ph{2}))$$
Then $\cat$ does not have $J$-powers due to the size issues --- distributors $\catl{A} \nrightarrow \catl{B}$ correspond to functors $\catl{A} \rightarrow \catw{Set}^{\catl{B}^{op}}$, but the category $\catw{Set}^{\catl{B}^{op}}$ usually is not small, nor even equivalent to a small one. Unfortunately, these size issues are fundamental --- there is no sensible restriction on the sizes of objects and morphism to make $\cat$ admit $J$-powers. However, \emph{some} of the distributors \emph{are} classified in such a way. These observations lead us to the concept of a Yoneda triangle.

\begin{definition}[Yoneda triangle]\label{d:yoneda:triangle}
Let $\catl{W}$ be a 2-category. A Yoneda triangle in $\catl{W}$, written $\yon{\eta}{y}{f}{g}$, consists of three morphisms $\mor{y}{A}{\overline{A}}$, $\mor{f}{A}{B}$ and $\mor{g}{B}{\overline{A}}$ together with a 2-morphism $\mor{\eta}{y}{g \circ f}$ which exhibits $g$ as a pointwise left Kan extension of $y$ along $f$, and exhibits $f$ as an absolute left Kan lifting\footnote{The concept of a Kan lifting is the \emph{opposite} of the concept of Kan extension --- i.e.~a Kan lifting in $\catl{W}$ is a Kan extension in $\catl{W}^{op}$.} of $y$ along $g$:
$$\bfig
\place(280, 400)[\twoar(-1, -1)\scriptstyle{\eta}]
\ptriangle(0, 0)/->`->`<-/<700, 500>[A`\overline{A}`B;y`f = \lif_g(y)`g = \lan_f(y)]
\efig$$
The absoluteness of a Kan lifting means that the lifting is preserved by any morphism $\mor{k}{K}{A}$ --- i.e.~the 2-morphism $\eta \circ k$ exhibits $\lif_g(y)\circ k$ as  $\lif_g(y \circ k)$.
\end{definition}
The idea of a Yoneda triangle is that, we have a morphism $\mor{y}{A}{\overline{A}}$ which plays the role of a ``defective identity'' and for a given morphism $\mor{f}{A}{B}$ we try to characterise its right adjoint up to the ``defective identity'' $y$.
\begin{example}[Adjuntion as Yoneda triangle]\label{e:adjunction:triangle}
A 1-morphism $\mor{f}{A}{B}$ in a 2-category $\catl{W}$ has a right adjoint $\mor{g}{B}{A}$ with unit $\mor{\eta}{\mathit{id}}{g \circ f}$ precisely when $\yon{\eta}{\mathit{id}}{f}{g}$ is a Yoneda triangle:
$$\bfig
\place(280, 400)[\twoar(-1, -1)\scriptstyle{\eta}]
\place(-130, 560)[\twoar(-1, -1)\scriptstyle{\epsilon}]
\node a1(0,500)[A]
\node a2(700,500)[A]
\node b2(0,1000)[B]
\node b(0,0)[B]
\arrow/->/[a1`a2;\mathit{id}]
\arrow|m|/->/[a1`b;f = \lif_g(\mathit{id})]
\arrow|r|/<-/[a2`b;g = \lan_f(\mathit{id})]
\arrow/->/[b2`a2;g]
\arrow/->/[b2`a1;g]
\arrow/{@{>}@/_3em/}/[b2`b;\mathit{id}]
\efig$$
Since $f = \lif_g(\mathit{id})$ is an absolute lifting, $f\circ g$ is a lifting of $g$ through $g$ with $\mor{\eta \circ g}{g}{g \circ f \circ g}$.  
By the universal property of the lifting, there is a unique 2-morphism $\mor{\epsilon}{f \circ g}{\mathit{id}}$ such that $g \circ \epsilon \bullet \eta \circ g = \mathit{id}$, which may be defined as the counit of the adjunction. We have to show that also the other triangle equality holds. Let us first postcompose the equation $g \circ \epsilon \bullet \eta \circ g = \mathit{id}$ with $f$ to obtain $g \circ \epsilon \circ f \bullet \eta \circ g \circ f = \mathit{id}$. Then postcompose it with $\eta$ to get $g \circ \epsilon \circ f \bullet \eta \circ g \circ f \bullet \eta = \eta$. But this equation, under bijection provided by $\lif_g(\id{})$, corresponds to the equation $\epsilon \circ f \bullet f \circ \eta = \id{}$, which is the required triangle equality.

On the other hand, let us assume that $f$ is left adjoint to $g$ with unit $\mor{\eta}{\id{}}{g \circ f}$ and counit $\mor{\epsilon}{f \circ g}{\id{}}$.
We shall see that for every $\mor{k}{C}{A}$ the composite $\eta \circ k$ exhibits $f \circ k$ as the left Kan lifting of $k$ along $g$:
$$\bfig
\place(260, 360)[\twoar(-1, -1)\scriptstyle{\eta \circ k}]
\place(-100, 300)[\twoar(-1, -1)\scriptstyle{\alpha}]
\node a1(0,500)[C]
\node a2(700,500)[A]
\node b(0,0)[B]
\arrow/->/[a1`a2;k]
\arrow|r|/->/[a1`b;f \circ k]
\arrow|r|/<-/[a2`b;g]
\arrow/{@{>}@/_2em/}/[a1`b;h]
\efig$$
We have to show that the assignments:
$$f \circ k \to^\alpha h \mapsto k \to^{g \circ \alpha \bullet \eta \circ k} g \circ h$$
and:
$$ k \to^{\beta} g \circ h \mapsto f \circ k \to^{\epsilon \circ h \bullet f \circ \beta} h$$
are inverse of each other. Let us check the first composition:
\begin{eqnarray*}
\epsilon \circ h \bullet f \circ (g \circ \alpha \bullet \eta \circ k) & = & \epsilon \circ h \bullet f \circ g \circ \alpha \bullet f \circ \eta \circ k\\
 &=& \alpha \bullet \epsilon \circ f \circ k \bullet f \circ \eta \circ k\\
&=& \alpha
\end{eqnarray*}
where the first and second equalities follow from the interchange law of a 2-category, and the last one follows by the triangle equation. Similarly we may check the second composition:
\begin{eqnarray*}
g \circ (\epsilon \circ h \bullet f \circ \beta) \bullet \eta \circ k &=& g \circ \epsilon \circ h \bullet g \circ f \circ \beta \bullet \eta \circ k \\
&=& g \circ \epsilon \circ h \bullet \eta \circ g \circ h \bullet \beta \\
&=& \beta
\end{eqnarray*}

The fact that $g$ is a pointwise left extension of $\id{}$ along $f$ follows from a more general observation that a left Kan extension along a left adjoint always exists and is pointwise (Proposition 20 in \cite{yonedaincat}). However, it is illustrative to see how the bijections defining Kan extensions are constructed in our particular case. Let us extend the diagram of adjunction $f \dashv g$ by generalised elements ${a_X \in A}$, ${b_X \in B}$:
$$\bfig
\node x(0,0)[X]
\node a(1000,500)[A]
\pullback(0, 0)[\comma{f}{b}`A`X`B;\pi_1`\pi_2`f`b_X]/{@{>}@/^1em/}`-->`{@{>}@/_1em/}/[X;g(b_X)`\exists ! h`\id{}]
\place(250, 250)[\twoar(-1, -1)\scriptstyle{\pi}]
\place(700, 330)[\twoar(-1, -1)\scriptstyle{\eta}]
\place(0, 690)[\twoar(-1, -1)\scriptstyle{\epsilon_b}]
\place(700, 0)[\twoar(-1, -1)\scriptstyle{\beta}]
\ptriangle(500, 0)/->``<-/<500, 500>[A`A`B;\id{}``g]
\arrow|m|/{@{>}@/_4em/}/[x`a;a_X]
\efig$$
where $\mor{h}{X}{\comma{f}{b}}$ is the unique morphism to the comma object induced by the counit $\mor{\epsilon_b}{f(g(b_X))}{b_X}$. Then, one part of the bijective correspondence is given by assigning to $\mor{\beta}{g(b_X)}{a_X}$ an arrow $\mor{\beta \circ \pi_2 \bullet g \circ \pi \bullet \eta \circ \pi_1}{\pi_1}{a_X \circ \pi_2}$, and the other is given by composition with $h$.
\end{example}
Generally, a Yoneda-like triangle $\yon{\eta}{y}{f}{g}$ where $g$ is not assumed to be the left Kan extension of $y$ along $f$ is called an adjunction relative to $y$ \cite{relativeadjunction}. Note however, that in such a case $g$ need not be uniquely determined by $f$.

Just like in \cite{log1} we provided an elementary description of pointwise Kan extensions, we shall now give a similar characterisation of absolute Kan liftings. Let us extend the diagram of a Yoneda triangle $\yon{\eta}{y}{f}{g}$, by taking generalised elements $a_X \in A, b_X \in B$ and a generalised arrow $f(a_X) \to^k b_X$: 
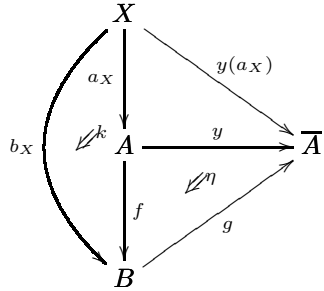
$$\bfig
\place(280, 400)[\twoar(-1, -1)\scriptstyle{\eta}]
\place(-130, 560)[\twoar(-1, -1)\scriptstyle{k}]
\node a1(0,500)[A]
\node a2(700,500)[\overline{A}]
\node b2(0,1000)[X]
\node b(0,0)[B]
\arrow/->/[a1`a2;y]
\arrow|r|/->/[a1`b;f]
\arrow|r|/<-/[a2`b;g]
\arrow/->/[b2`a2;y(a_X)]
\arrow/->/[b2`a1;a_X]
\arrow/{@{>}@/_3em/}/[b2`b;b_X]
\efig$$
The absoluteness of a left Kan lifting says that there is a bijective correspondence:
\begin{center}
\begin{tabular}{ccc}
$f(a_X)$ & $\overset{k}\rightarrow$ & $b_X$\\
\hline\hline
$y(a_X)$ & $\overset{\eta_a \circ k}\rightarrow$ & $g(b_X)$\\
\end{tabular}
\end{center}
which clearly reassembles the usual $\hom$-definition of adjunction on generalised elements. Moreover, using the formula for pointwise left Kan extension, we may write:
$$g(b_X) = \underset{f(a_X) \rightarrow b_X}\coprod y(a_X)$$
which in case of $\onecat$ may be interpreted as the colimit of $y$ taken over comma category $\comma{f}{\Delta_{b_X}}$. 
In particular, we have the following characterisation of Yoneda triangles in $\onecat$.
\begin{example}[Yoneda triangles in $\onecat$]
If we take $\catl{W}$ to be the 2-category $\catw{Cat}$ of locally small categories, functors and natural transformations, then the condition that $G$ is a pointwise left Kan extension of $\mor{Y}{\catl{A}}{\overline{\catl{A}}}$ along $\mor{F}{\catl{A}}{\catl{B}}$ reduces to:
$$G(\ph{}) = \int^{A\in\catl{A}} \hom_\catl{B}(F(A), \ph{}) \times Y(A)$$
In case category $\overline{\catl{A}}$ is not tensored over $\catw{Set}$ the above coend has to be interpreted as the colimit of $Y$ weighted by:
$${\hom_\catl{B}(F(\ph{1}), \ph{2})}$$
The condition that $F$ is an absolute left Kan lifting of $Y$ along $G$ reduces to:
$$\hom_\catl{B}(F(\ph{1}), \ph{2}) \approx \hom_{\overline{\catl{A}}}(Y(\ph{1}), G(\ph{2}))$$

Furthermore, if $Y$ is dense, then $G$ is automatically a pointwise Kan extension in a canonical way --- from density we have:
$$G(\ph{}) \approx \int^{A\in\catl{A}} \hom(Y(A), G(\ph{})) \times Y(A)$$
and using the formula for an absolute lifting: 
$$G(\ph{}) \approx \int^{A\in\catl{A}} \hom(F(A), \ph{}) \times Y(A)$$
\end{example}
This example needs more elaboration. In the literature, there exist two essentially different notions of pointwise Kan extensions. The older, provided by Eduardo Dubuc \cite{dubuckan} for enriched categories, defines pointwise Kan extensions as appropriate enriched (co)ends:
\begin{eqnarray*}
\ran_F(Y) &=& \int_{A \in \catl{A}} Y(A)^{\hom_\catl{B}(\ph{}, F(C))} \\
\lan_F(Y) &=& \int^{A \in \catl{A}} \hom_\catl{B}(F(C), \ph{}) \otimes Y(A)
\end{eqnarray*}
The newer, provided by Ross Street \cite{yonedaincat}, works in the general context of (sufficiently complete) 2-categories, and has been used in the paper up to this point. As mentioned in \cite{yonedaincat} \cite{kelly} these definitions agree for categories enriched in $\catw{Set}$, and for categories enriched in the 2-valued Boolean algebra $2$, but Street's definition is stronger than Dubuc's one for general enriched categories (it is strictly stronger for categories enriched in abelian groups $\catw{Ab}$, and for categories enriched in $\catw{Cat}$). Steve Lack \cite{catcomp} blamed for this mismatch the definition of a category of $\catl{C}$-enriched categories, which ``can't see'' the extra structure of a $\catl{C}$-enriched category on functor categories $\hom(A, B)$. Although it is certainly true that the category $\catw{Cat}(\catl{C})$ of $\catl{C}$-enriched categories is more than a 2-category --- after all, it is a $\catw{Cat}(\catl{C})$-enriched category with an underlying 2-category --- the reasoning is not correct. Technically, the reasoning cannot be right, because treating a 2-category  as a $\catw{Cat}(\catl{C})$-enriched category and carrying to this setting Street's definition of pointwise Kan extension may only strengthen the concept of a Kan extension, which is, actually, in its ordinary 2-categorical form, stronger than Dubuc's one. More importantly, also philosophically the reasoning cannot be right --- the enrichment of $\catw{Cat}(\catl{C})$ in $\catw{Cat}(\catl{C})$ is a self-enrichment, which means that it is completely recoverable from its underlying 2-category; the idea behind Street's pointwise Kan extensions was to define the Kan extension at ``every generalised 2-point'' just to evade defining it on ``enriched objects''.

\begin{example}[Yoneda triangle along Yoneda embedding]\label{e:yoneda:along:yoneda}
For any functor $\mor{F}{\catl{A}}{\catl{B}}$ between locally small categories, there is a Yoneda triangle:
$$\bfig
\place(280, 400)[\twoar(-1, -1)\scriptstyle{\eta}]
\ptriangle(0, 0)/->`->`<-/<700, 500>[\catl{A}`\catw{Set}^{\catl{A}^{op}}`\catl{B};y_\catl{A}`F`\hom_\catl{B}(F(\ph{2}),\ph{1})]
\efig$$
which reassembles the fact that every functor always has a ``distributional'' right adjoint\footnote{Every functor has a right adjoint in the weak 2-category of distributors.}. The same is true for internal categories and for categories enriched in a cocomplete symmetric monoidal closed category, and generally (almost by definition) for any 2-category equipped with a Yoneda structure in the sense of \cite{yoneda}.
\end{example}
The essence of the example is that because the Yoneda functor $\mor{y_\catl{B}}{\catl{B}}{\catw{Set}^{\catl{B}^{op}}}$ is a full and faithful embedding, functors $\mor{F}{\catl{A}}{\catl{B}}$ may be thought as of distributors 
$$y_\catl{B} \circ F = \hom_\catl{B}(\ph{2}, F(\ph{1}))$$
Every distributor arisen in this way has a right adjoint distributor ${\hom_\catl{B}(F(\ph{2}), \ph{1})}$ in the (weak) 2-category of distributors. The distributor ${\hom_\catl{B}(F(\ph{2}), \ph{1})}$ actually has type $\catl{B} \rightarrow \catw{Set}^{\catl{A}^{op}}$, which is the only reason that may prevent $F$ from having the ordinary (functorial) right adjoint $\mor{G}{\catl{B}}{\catl{A}}$. Formally, we say that $F$ has a right adjoint, if there exists $G$ such that:
$$y_\catl{A} \circ G \approx \hom_\catl{B}(F(\ph{2}), \ph{1})$$
which means:
$$\hom_\catl{A}(\ph{2}, G(\ph{1})) \approx \hom_\catl{B}(F(\ph{2}), \ph{1})$$

Of course, a Yoneda 2-triangle is a Yoneda triangle in the (2-)category of 2-categories, 2-functors, and 2-natural transformations. However, in the light of our elaboration on ``pointwiseness'', we shall weaken the definition of pointwise Kan extension to the one suitable for enriched categories --- as it is much easier and convenient to work with. 
\begin{definition}[Yoneda bitriangle]\label{d:yoneda:bitriangle}
A Yoneda bitriangle $\yon{\eta}{Y}{F}{G}$ consists of pseudofunctors $\mor{Y}{\catl{A}}{\overline{\catl{A}}}$, $\mor{F}{\catl{A}}{\catl{B}}$, $\mor{G}{\catl{B}}{\overline{\catl{A}}}$ between (weak) 2-categories $\catl{A}, \overline{\catl{A}}, \catl{B}$ and a pseudonatural transformation $\mor{\eta}{Y}{G \circ F}$ that induces natural equivalences between functors:
$$\hom_{\overline{\catl{A}}}(Y(\ph{2}), G(\ph{1})) \approx \hom_\catl{B}(F(\ph{2}), \ph{1})$$
and between functors:
$$\hom_{\overline{\catl{A}}}(G(\ph{1}), \ph{2}) \approx \hom(\hom_\catl{B}(F(\ph{3}), \ph{1}), \hom_{\overline{\catl{A}}}(Y(\ph{3}), \ph{2}))$$
\end{definition}
The last equivalence should be informally understood as the following ``equivalence''\footnote{It may be expressed as such an equivalence in case objects on the right hand side are well-defined.}:
$$G(\ph{}) \approx \int^{A\in\catl{A}} \hom_\catl{B}(F(A), \ph{}) \times Y(A)$$
Observe that even in case $\mor{Y}{\catl{A}}{\overline{\catl{A}}}$ is only weakly 2-dense, i.e.~there is a canonical equivalence:
$$\hom(\hom_{\overline{\catl{A}}}(Y(\ph{1}), \ph{2}), \hom_{\overline{\catl{A}}}(Y(\ph{1}), \ph{3})) \approx \hom_{\overline{\catl{A}}}(\ph{2}, \ph{3})$$
if $G$ satisfies the first condition, then it automatically satisfies the second as well:
\begin{eqnarray*}
&&\hom(\hom_\catl{B}(F(\ph{3}), \ph{1}), \hom_{\overline{\catl{A}}}(Y(\ph{3}), \ph{2}))\\
& \approx &\hom(\hom_{\overline{\catl{A}}}(Y(\ph{3}), G(\ph{1})), \hom_{\overline{\catl{A}}}(Y(\ph{3}), \ph{2}))\\
& \approx &\hom_{\overline{\catl{A}}}(G(\ph{1}), \ph{2})\\
\end{eqnarray*}
 
We shall be mostly interested in Yoneda bitriangles arisen from proarrow equipment \cite{proarrows} \cite{proarrows2}, which we recall below. Let $\mor{J}{\catl{A}}{\catl{B}}$ be a (weak) 2-functor from a (strict) 2-category $\catl{A}$ to a (weak) 2-category  $\catl{B}$. We say that $J$ equips $\catl{A}$ with proarrows if the following holds:
\begin{itemize}
\item $J$ is bijective on objects
\item $J$ is locally fully faithful, which means that for every pair of objects $X, Y \in \catl{A}$ the induced functor $\hom_\catl{A}(X,Y) \rightarrow \hom_\catl{B}(J(X), J(Y))$ is fully faithful
\item for every 1-morphism $\mor{f}{X}{Y}$ in $\catl{A}$ the corresponding morphism $\mor{J(f)}{J(X)}{J(Y)}$ in $\catl{B}$ has a right adjoint
\end{itemize}
A proarrow equipment reassembles the concept of an allegory in a 2-dimensional context.
\begin{definition}[2-powers]\label{d:2:powers}
Let $\mor{J}{\catl{A}}{\catl{B}}$ be an equipment of $\catl{A}$ with proarrows, and $\mor{Y}{\catl{A}}{\overline{\catl{A}}}$ a 2-functor making $\catl{A}$ a full 2-subcategory of $\overline{\catl{A}}$. Then $\catl{A}$ has $J$-relative 2-powers if $J$ and $Y$ can be completed to a Yoneda bitriangle $\yon{\eta}{Y}{J}{P}$ with $\mor{P}{\catl{B}}{\overline{\catl{A}}}$ and $\mor{\eta}{Y}{P \circ J}$.
\end{definition}  
\begin{example}[Categorical 2-powers]\label{e:cat:triangle}
The archetypical situation is when we take $\yon{\eta}{\mor{Y}{\catw{cat}}{\catw{Cat}}}{\mor{J}{\catw{cat}}{\catw{dist}}}{\mor{P}{\catw{dist}}{\catw{Cat}}}$, where $\catw{cat}$ is the 2-category of small categories, $\catw{Cat}$ is the 2-category of locally small categories, and $\catw{dist}$ is the (weak) 2-category of distributors between small categories. Then $J \colon \catw{cat} \rightarrow \catw{dist}$, $Y \colon \catw{cat} \rightarrow \catw{Cat}$ are the usual embeddings, $P \colon \catw{dist} \rightarrow \catw{Cat}$ is the covariant 2-power pseudofunctor $\catw{Set}^{(-)^{op}}$ defined on distributors via left Kan extensions, and $\eta_\catl{A} \colon \catl{A} \rightarrow \catw{Set}^{\catl{A}^{op}}$ is the Yoneda embedding of a small category $\catl{A}$.
There are isomorphisms of categories:
$$\hom_{\catw{dist}}(\catl{A}, \catl{B}) \approx \hom_{\catw{Cat}}(\catl{A}, \catw{Set}^{\catl{B}^{op}})$$
where $\catl{A}$ and $\catl{B}$ are small. Therefore, to show that $P$ is a (weak) pointwise left Kan extension it suffices to show that $Y$ is 2-dense. However, $Y$ is obviously 2-dense, because the terminal category $1$ is a 2-dense subcategory of $\catw{Cat}$ and $Y$ is fully faithful.
\end{example}
Here is a similar result for internal categories.

\begin{theorem}[$\catl{C}$-internal 2-powers]\label{e:cat:internal:triangle}
Let $\catl{C}$ be a finitely cocomplete locally cartesian closed category. There is a Yoneda bitriangle: $$\yon{\eta}{\mor{\mathit{fam}}{\catw{cat}(\catl{C})}{\catw{Cat}^{\op{\catl{C}}}}}{\mor{J}{\catw{cat}(\catl{C})}{\catw{dist}(\catl{C})}}{\mor{P}{\catw{dist}(\catl{C})}{\catw{Cat}^{\op{\catl{C}}}}}$$
where $\catw{cat}(\catl{C})$ is the 2-category of $\catl{C}$-internal categories, $\catw{dist}(\catl{C})$ is the (weak) 2-category of $\catl{C}$-internal distributors with $J$ the usual embedding, and:
$$\mor{\word{fam}}{\catw{cat}(\catl{C})}{\catw{Cat}^{\op{\catl{C}}}}$$ is the  
canonical family functor (the externalisation functor). Pseudofunctor: $$\mor{P}{\catw{dist}(\catl{C})}{\catw{Cat}^{\op{\catl{C}}}}$$ is given by:
\begin{eqnarray*}
P(A) &=& \word{fam}(\catl{C})^{\op{\word{fam}(A)}}\\
P(A \overset{F}\nrightarrow B) &=& \lan_{y_A}(F) 
\end{eqnarray*}
where $\word{fam}(\catl{C})$ is a split indexed category corresponding to the fundamental (i.e.~codomain) fibration, and:
$$\mor{y_A}{\word{fam}(A)}{\word{fam}(\catl{C})^{\op{\word{fam}(A)}}}$$
is the usual internal Yoneda embedding defined as the cartesian transposition of:
$$\mor{\underline{\hom(\ph{2},\ph{1})}}{\word{fam}(A) \times \op{\word{fam}(A)}}{\word{fam}(\catl{C})}$$
\end{theorem}
\begin{proof}
By universal properties of Kan extensions, $P$ is a pseudofunctor $\catw{dist}(\catl{C}) \to \catw{Cat}^{\op{\catl{C}}}$. There is an equivalence of categories\footnote{In fact this equivalence is almost a definition of the category $\catw{dist}(\catl{C})$.} (Section 4 of \cite{variations}, Section 3 of \cite{cohomology}):
$$\hom_{\catw{dist}(\catl{C})}(A, B) \approx \hom_{\catw{Cat}^\op{\catl{C}}}(\word{fam}(A), \word{fam}(\catl{C})^{\op{\word{fam}(B)}})$$
To show that $P$ is a (pointwise) left Kan extension it suffices to show that $\mathit{fam}$ is 2-dense. However, $\mathit{fam}$ on discrete internal categories is clearly 2-dense by (weak) 2-Yoneda lemma, and discrete internal categories form a full 2-subcategory of all categories. Therefore $\mathit{fam}$ is 2-dense.
\end{proof}

It requires much more work to obtain analogical result for enriched categories. The difficulty is of the same kind as we encountered earlier --- discrete objects in the category of enriched categories are generally not dense (more --- they rarely constitute a generating family) and there is no canonical candidate for any subcategory giving a dense notion of discreteness.
First, let us observe that every enriched category is a canonical limit over its full subcategories consisting of at most three objects.

\begin{lemma}[On a 2-dense subcategory of $\catw{Cat}(\catl{C})$]\label{l:dense:enriched}
Let $\tuple{I, \otimes, \catl{C}}$ be a complete and cocomplete symmetric monoidal closed category. The category of small $\catl{C}$-enriched categories is a 2-dense subcategory of all $\catl{C}$-enriched categories.
\end{lemma}
\begin{proof}
We have to show that the following categories of natural transformations are isomorphic in a canonical way for all $\catl{C}$-enriched categories $A, B \in \catw{Cat}(\catl{C})$:
\begin{center}
\begin{tabular}{p{\linewidth}}
$\hom_{\onecat^{\op{\catw{Cat}(\catl{C})}}}(\hom_{\catw{Cat}(\catl{C})}(-, A), \hom_{\catw{Cat}(\catl{C})}(-, B))\approx$\\
\multicolumn{1}{r}{$\hom_{\onecat^{\op{\catw{Cat}_S(\catl{C})}}}(\hom_{\catw{Cat}(\catl{C})}(Y(-), A), \hom_{\catw{Cat}(\catl{C})}(Y(-), B))$}
\end{tabular}
\end{center}
where $\mor{Y}{\catw{Cat}_S(\catl{C})}{\catw{Cat}(\catl{C})}$ is the embedding of small categories $\catw{Cat}_S(\catl{C})$ into all (locally small) categories $\catw{Cat}(\catl{C})$. To simplify the proof, let us observe that it suffices to show that the underlying sets of the above natural transformation objects are naturally bijective (i.e.~that $\catw{Cat}_S(\catl{C})$ is 1-dense in $\catw{Cat}(\catl{C})$). Because $\catw{Cat}(\catl{C})$ is cotensored over small categories, we have natural in $X \in \cat$ bijections:
\begin{center}
\begin{tabular}{p{\linewidth}}
$\hom_{\onecat^{\op{\catw{Cat}(\catl{C})}}}(\hom_{\catw{Cat}(\catl{C})}(-, A), \hom_{\catw{Cat}(\catl{C})}(-, X \pitchfork B)) \approx $\\
$\hom_{\onecat^{\op{\catw{Cat}(\catl{C})}}}(\hom_{\catw{Cat}(\catl{C})}(-, A), \hom_{\catw{Cat}(\catl{C})}(-, B)^X) \approx$ \\
\multicolumn{1}{r}{$\hom_\onecat(X, \hom_{\onecat^{\op{\catw{Cat}(\catl{C})}}}(\hom_{\catw{Cat}(\catl{C})}(-, A), \hom_{\catw{Cat}(\catl{C})}(-, B)))$}
\end{tabular}
\end{center}
and similarly --- because $\catw{Cat}_S(\catl{C})$ is cotensored over small categories:
\begin{center}
\begin{tabular}{p{\linewidth}}
$\hom_{\onecat^{\op{\catw{Cat}_S(\catl{C})}}}(\hom_{\catw{Cat}(\catl{C})}(Y(-), A), \hom_{\catw{Cat}(\catl{C})}(Y(-), X \pitchfork B)) \approx $\\
$\hom_{\onecat^{\op{\catw{Cat}_S(\catl{C})}}}(\hom_{\catw{Cat}(\catl{C})}(Y(-), A), \hom_{\catw{Cat}(\catl{C})}(Y(-), B)^X) \approx$ \\
\multicolumn{1}{r}{$\hom_\onecat(X, \hom_{\onecat^{\op{\catw{Cat}_S(\catl{C})}}}(\hom_{\catw{Cat}(\catl{C})}(Y(-), A), \hom_{\catw{Cat}(\catl{C})}(Y(-), B)))$}
\end{tabular}
\end{center}
By the usual argument, categories:
$$\hom_{\onecat^{\op{\catw{Cat}(\catl{C})}}}(\hom_{\catw{Cat}(\catl{C})}(-, A), \hom_{\catw{Cat}(\catl{C})}(-, B))$$
and:
$$\hom_{\onecat^{\op{\catw{Cat}_S(\catl{C})}}}(\hom_{\catw{Cat}(\catl{C})}(Y(-), A), \hom_{\catw{Cat}(\catl{C})}(Y(-), B))$$
are isomorphic iff the sets:
$$\hom_{\catw{SET}}(X, \hom_{\onecat^{\op{\catw{Cat}(\catl{C})}}}(\hom_{\catw{Cat}(\catl{C})}(-, A), \hom_{\catw{Cat}(\catl{C})}(-, B)))$$
and:
$$\hom_{\catw{SET}}(X, \hom_{\onecat^{\op{\catw{Cat}_S(\catl{C})}}}(\hom_{\catw{Cat}(\catl{C})}(Y(-), A), \hom_{\catw{Cat}(\catl{C})}(Y(-), B)))$$
are naturally bijective in $X \in \cat$. Therefore, if the canonical function:
\begin{center}
\begin{tabular}{c}
$\hom_{\onecat^{\op{\catw{Cat}(\catl{C})}}}(\hom_{\catw{Cat}(\catl{C})}(-, A), \hom_{\catw{Cat}(\catl{C})}(-, X \pitchfork B))$\\
$\downarrow$\\
$\hom_{\onecat^{\op{\catw{Cat}_S(\catl{C})}}}(\hom_{\catw{Cat}(\catl{C})}(Y(-), A), \hom_{\catw{Cat}(\catl{C})}(Y(-), X \pitchfork B))$
\end{tabular}
\end{center}
is a bijection, then the canonical functor:
\begin{center}
\begin{tabular}{c}
$\hom_{\onecat^{\op{\catw{Cat}(\catl{C})}}}(\hom_{\catw{Cat}(\catl{C})}(-, A), \hom_{\catw{Cat}(\catl{C})}(-, B))\approx$\\
$\downarrow$ \\
$\hom_{\onecat^{\op{\catw{Cat}_S(\catl{C})}}}(\hom_{\catw{Cat}(\catl{C})}(Y(-), A), \hom_{\catw{Cat}(\catl{C})}(Y(-), B))$
\end{tabular}
\end{center}
is an isomorphism.

Denote by $\catw{Cat}_3(\catl{C})$ the full 1-subcategory of $\catw{Cat}(\catl{C})$ consisting of categories having at most three objects, and by $\mor{K}{\catw{Cat}_3(\catl{C})}{\catw{Cat}(\catl{C})}$ its embedding. We show that $\catw{Cat}_3(\catl{C})$ is a 1-dense subcategory of $\catw{Cat}(\catl{C})$, which by fully-faithfulness of $Y$ implies that $\catw{Cat}_S(\catl{C})$ is 1-dense subcategory of $\catw{Cat}(\catl{C})$, and by the above that it is 2-dense. 

The direction showing that the canonical morphism is injective is easy --- if $\mor{\alpha}{\hom_{\onecat(\catl{C})}(-, A)}{\hom_{\onecat(\catl{C})}(-, B)}$ is a natural transformation, then its restriction $\mor{\widetilde{\alpha}}{\hom_{\onecat(\catl{C})}(K(-), A)}{\hom_{\onecat(\catl{C})}(K(-), B)}$ to a subcategory is natural as well, and since $\catw{Cat}_3(\catl{C})$ is clearly a generating subcategory, then this assignment is injective. So let us focus on the other direction.

Observe that every $\catl{C}$-enriched category $A$ may be canonically represented as a colimit over at most three-object categories:
\begin{itemize}
\item for every triple of objects $X, Y, Z \in A$, let $A_{X, Y, Z}$ be the full subcategory of $A$ on this triple with injection $\mor{j_{X,Y,Z}^A}{A_{X, Y, Z}}{A}$; similarly define $\mor{j_{X,Y}^{X,Y,Z}}{A_{X, Y}}{A_{X,Y,Z}}$ for the full subcategory $A_{X, Y}$ of $A_{X, Y, Z}$ on every pair $X, Y \in A_{X, Y, Z}$, and $\mor{j_{X}^{X,Y}}{A_X}{A_{X, Y}}$ for the full one-object subcategory on every object $X \in A_{X, Y}$
\item for diagram $D_A$ consisting of all such defined injections   $\mor{j_{X,Y}^{X,Y,Z}}{A_{X, Y}}{A_{X,Y,Z}}$, $\mor{j_{X}^{X,Y}}{A_X}{A_{X, Y}}$, the category $A$ together with  $\mor{j_{X,Y,Z}^A}{A_{X, Y, Z}}{A}$ is the colimit of $D_A$ --- if $B$ is another category with cocone $\mor{\tau_{X,Y,Z}^A}{A_{X, Y, Z}}{B}$ then the unique functor $\mor{H}{A}{B}$ is given on objects by $H(X) = (\tau_{X,Y,Z}^A \circ j_{X,Y}^{X,Y,Z} \circ j_{X}^{X,Y})(X)$, and similarly on morphisms; the compositions are preserved by $H$, because they are preserved pairwise by each $\tau_{X,Y,Z}^A$, and preservation of identities is obvious. 
\end{itemize}
Let $\mor{\widetilde{\alpha}}{\hom_{\onecat(\catl{C})}(K(-), A)}{\hom_{\onecat(\catl{C})}(K(-), B)}$ be a natural transformation. By naturality, the diagram $D_A$ is mapped by $\widetilde{\alpha}$ to a cocone under $B$. By universal property of colimits, this cocone induces a morphism $\mor{c}{A}{B}$, which by Yoneda lemma is tantamount to the natural transformation:
$$\mor{\hom_{\onecat(\catl{C})}(-, c)}{\hom_{\onecat(\catl{C})}(-, A)}{\hom_{\onecat(\catl{C})}(-, B)}$$
We have to show that $\hom_{\onecat(\catl{C})}(-, c)$ on $\catw{Cat}_3(\catl{C})$ is equal to $\widetilde{\alpha}$, that is: for any at most three-element category $M$ and a functor $\mor{f}{M}{A}$ the composite $c \circ f$ is equal to $\widetilde{\alpha}(f)$. But this is easy. Let us assume that $M$ has exactly three objects $X, Y, Z$ then $\mor{f}{M}{A}$ factors as  $\mor{g}{M}{A_{f(X),f(Y),f(Z)}}$ through injection $\mor{j_{f(X),f(Y),f(Z)}^A}{A_{f(X),f(Y),f(Z)}}{A}$. By naturality of $\widetilde{\alpha}$ we have: $\widetilde{\alpha}(j_{f(X),f(Y),f(Z)}^A) \circ g = \widetilde{\alpha}(j_{f(X),f(Y),f(Z)}^A) \circ g) = \widetilde{\alpha}(f)$ and by the definition: $c \circ j_{f(X),f(Y),f(Z)}^A = \widetilde{\alpha}(j_{f(X),f(Y),f(Z)}^A)$. Therefore $c \circ f = \widetilde{\alpha}(f)$. A similar argument exhibits equality between components of natural transformations on less than three object categories.
\end{proof}
\begin{theorem}[$\catl{C}$-enriched 2-powers]\label{e:cat:enriched:triangle}
Let $\tuple{I, \otimes, \catl{C}}$ be a complete and cocomplete symmetric monoidal closed category. There is a Yoneda bitriangle:
$$\yon{\eta}{\mor{Y}{\catw{Cat}_S(\catl{C})}{\catw{Cat}(\catl{C})}}{\mor{J}{\catw{Cat}_S(\catl{C})}{\catw{Dist}(\catl{C})}}{\mor{P}{\catw{Dist}(\catl{C})}{\catw{Cat}(\catl{C})}}$$
where $\catw{Cat}_S(\catl{C})$ is the 2-category of small $\catl{C}$-enriched categories, $\catw{Cat}(\catl{C})$ is the 2-category of all (i.e.~locally small) $\catl{C}$-enriched categories, $\catw{Dist}(\catl{C})$ is the (weak) 2-category of $\catl{C}$-enriched distributors between small categories, and $J, Y$ are the canonical embeddings. Pseudofunctor: $$\mor{P}{\catw{Dist}(\catl{C})}{\catw{Cat}(\catl{C})}$$ is given by:
\begin{eqnarray*}
P(A) &=& \catl{C}^{\op{A}}\\
P(A \overset{F}\nrightarrow B) &=& \lan_{y_A}(F) 
\end{eqnarray*}
where $\mor{y_A}{A}{\catl{C}^{\op{A}}}$ is the enriched Yoneda functor.
\end{theorem}
\begin{proof}
By definition of $\catw{Dist}(\catl{C})$ there is an equivalence of categories:
$$\hom_{\catw{Dist}(\catl{C})}(\catl{A}, \catl{B}) \approx \hom_{\onecat(\catl{C})}(\catl{A}, \catl{C}^{\op{\catl{B}}})$$
By Lemma \ref{l:dense:enriched} category $\catw{Cat}_S(\catl{C})$ is a 2-dense subcategory of $\catw{Cat}(\catl{C})$; therefore $P$ is a pointwise left Kan extension of $Y$ along $G$.
\end{proof}
It should be noted that the proarrow equipments in the above examples are canonically determined by the 2-categories of internal and enriched categories respectively --- in fact the categories of distributors are equivalent to the (weak) 2-categories of codiscrete cofibred spans (Theorem 14 in \cite{variations}) in these categories. One can seek a characterisation of a 2-topos along this line, but we leave it for a careful reader, as it is mostly irrelevant for our considerations.

\section{Power semantics} \label{s:power:semantics}
If ${\models} \subseteq {M \times S}$ is a binary relation between two sets: $M$, which is thought of as a set of models, and $S$, which is thought of as a set of syntactic elements (sentences), then we have ``for free'' Boolean semantics for propositional connectives formed over set $S$:
\begin{center}
\begin{tabular}{lcl}
$M \models \top$ &iff& true\\
$M \models \bot$ &iff& false\\
$M \models x \wedge y$ &iff& $M \models x \;\;\;\text{and}\;\;\; M \models y$\\
$M \models x \vee y$ &iff& $M \models x \;\;\;\text{and}\;\;\; M \models y$\\
$M \models x \Rightarrow y$ &iff& $M \models x \;\;\;\text{implies}\;\;\; M \models y$\\
\end{tabular}
\end{center}
More generally, in any topos with a subobject classifier $\Omega$, a relation ${\models} \subseteq {M \times S}$ corresponds to a morphism $\mor{\nu}{S}{\Omega^M}$. Since for every object $M$ the power object $\Omega^M$ inherits an internal Heyting algebra structure from $\Omega$, we may give the valuation semantics for propositional connectives in $S$ via the composition:
\begin{center}
\begin{tabular}{lcl}
$\nu(\top)$ &$\approx$& $\top$ \\
$\nu(\bot)$ &$\approx$& $\bot$ \\
$\nu(x \wedge y)$ &$\approx$& $\nu(x) \wedge \nu(y)$ \\
$\nu(x \vee y)$ &$\approx$& $\nu(x) \vee \nu(y)$ \\
$\nu(x \Rightarrow y)$ &$\approx$& $\nu(x) \Rightarrow \nu(y)$ \\
\end{tabular}
\end{center}
where $x, y \in S$ are generalised elements. The above should be read as follows --- given any generalised elements $X \to^{x, y} S$ there is a diagram:
$$\bfig
\node x(0,0)[X]
\node s(400,0)[S]
\node p(800,0)[\Omega^M]
\arrow/{@{>}@/_1em/}/[x`s;y]
\arrow/{@{>}@/^1em/}/[x`s;x]
\arrow[s`p;\nu]
\efig$$
then the semantics of meta-formula ``$x \wedge y$'' is:
$$\wedge \circ (\nu \circ x \times \nu \circ y)$$
where $\Omega^M \times \Omega^M \to^\wedge \Omega^M$ is the internal conjunction morphism in internal Heyting algebra $\Omega^M$; similarly for the other connectives.  
\begin{example}[Free propositional semantics]
Let us start with a set $\word{Var}$ and the equality relation  ${=} \subseteq {\word{Var} \times \word{Var}}$. Since every set is isomorphic to a coproduct on singletons, all generalised elements of a set are recoverable from its global elements. Therefore, we may restrict our semantics to global elements only. For every pair of elements $x, y \in \word{Var}$ the free semantics for the meta-conjunction $x \wedge y$ is $\nu(x) \wedge \nu(y) = v \mapsto (x = v) \wedge (y = v)$, and similarly for other connectives. Observe that this gives semantics for a pair $x, y \in \word{Var}$ interpreted as conjunction $x \wedge y$, without saying what exactly $x \wedge y$ is. If one is not comfortable with such semantics then one may ``materialize'' elements by forming an initial algebra. Formally, for a given set $\word{Var}$ let us define an endofunctor on $\catw{Set}$:
$$\word{F}(X) = (X \times X) \sqcup (X \times X) \sqcup (X \times X) \sqcup 1 \sqcup 1$$
and $\word{Prop}_\word{Var}$ as  the initial algebra for $F(X) \sqcup \word{Var}$. Now, the free semantics of ${=} \subseteq {\word{Var} \times \word{Var}}$ may be extended to the semantics for $\word{Prop}_\word{Var}$ via the unique morphism from the initial algebra to the algebra:
$$(2^\word{Var} \times 2^\word{Var}) \sqcup (2^\word{Var} \times 2^\word{Var}) \sqcup (2^\word{Var} \times 2^\word{Var}) \sqcup 1 \sqcup 1 \sqcup \mathit{Var} \to^{[\wedge,\vee,\Rightarrow,\top,\bot, =]} 2^\word{Var}$$ 
\end{example}
Much more is true. Not only does the power object $\Omega^M$ have all propositional connectives, in a sense, which we make precise in this section, $\Omega^M$ has all possible connectives.
\begin{example}[Relational semantics in $\catw{Set}$]
Let ${r} \subseteq {M \times M \times M}$ be a ternary relation on a set $M$. Then there is a corresponding binary operation $\otimes_r$ on $\Omega^M$ defined as follows:
$$f \otimes_r g = \lambda x \mapsto \underset{a, b \in M}{\exists} f(a)  \wedge g(b) \wedge r(x, a, b)$$
Moreover, $\otimes_r$ has ``exponentiations'' on each of its coordinates. They are given by the following formulae:
\begin{eqnarray*}
f \overset{L}\multimap_r g &=& \lambda a \mapsto \underset{x, b \in M}\forall f(b) \wedge r(x, a, b) \Rightarrow g(x) \\
f \overset{R}\multimap_r g &=& \lambda b \mapsto \underset{x, a \in M}\forall f(a) \wedge r(x, a, b) \Rightarrow g(x)
\end{eqnarray*}
We get the usual propositional connectives by considering relations associated to the unique comonoid structure $\tuple{\mor{!}{M}{1}, \mor{\Delta}{M}{M \times M}}$ in a cartesian closed category $\catw{Set}$ --- for $\mor{\phi, \psi}{M}{2}$: 
\begin{eqnarray*}
(\phi \wedge \psi)(x) &\mathit{iff} & \underset{a, b \in M}{\exists} \phi(a)  \wedge \psi(b) \wedge \tuple{x,x} = \tuple{a, b}\\
&\mathit{iff}& \phi(x)  \wedge \psi(x)
\end{eqnarray*}
and:
\begin{eqnarray*}
(\phi \overset{L}\Rightarrow \psi)(a) &\mathit{iff} & \underset{x, b \in M}\forall \phi(b) \wedge \tuple{x,x} = \tuple{a, b} \Rightarrow \psi(x) \\
&\mathit{iff} &  \phi(a) \Rightarrow \psi(a) \\
&\mathit{iff} & (\phi \overset{R}\Rightarrow \psi)(a)
\end{eqnarray*}
\end{example}
One may recognise in the above example the concept of ternary frame semantics for substructural logics \cite{ternarysemantics}.  The crucial point, however, is that such defined semantics have 2-dimensional analogues. The next example was the subject of Brain Day's thesis \cite{day}.

\begin{example}[Day convolution]\label{e:day}
Let $\tuple{\catl{C}, \otimes, I}$ be a complete and cocomplete symmetric monoidal closed category. Suppose $\dist{M}{\catl{A} \otimes \catl{A}}{\catl{A}}$ is a $\catl{C}$-enriched distributor. The convolution of $M$ is a functor $\mor{\otimes_M}{\catl{C}^{\op{\catl{A}}} \otimes \catl{C}^{\op{\catl{A}}}}{\catl{C}^{\op{\catl{A}}}}$ defined by the coeand:
$$(F \otimes_M G)(-) = \int^{B, C \in \catl{A}} F(B) \otimes G(C) \otimes M(-, B, C)$$
If $\tuple{\catl{A}, \dist{M}{\catl{A} \otimes \catl{A}}{\catl{A}}, \mor{J}{\op{\catl{A}}}{\catl{C}}}$ is a $\catl{C}$-promonoidal category (i.e.~a weak monoid in a (weak) 2-category of $\catl{C}$-enriched distributors (Chapter 2 of \cite{day})). The induced by convolution operation on $\catl{C}^{\op{\catl{A}}}$  yields a monoidal structure $\tuple{\catl{C}^{\op{\catl{A}}}, \otimes_M, J}$ on $\catl{C}^{\op{\catl{A}}}$. First observe that $J$ is the right unit of $\otimes_M$:
\begin{eqnarray*}
(F \otimes_M J)(-) & = & \int^{B, C \in \catl{A}} F(B) \otimes J(C) \otimes M(-, B, C)  \\
& \approx & \int^{B \in \catl{A}} F(B) \otimes \hom_\catl{A}(-, B) \\
& \approx & F(-)
\end{eqnarray*}
where $\int^{C \in \catl{A}} J(C) \otimes M(-, B, C) \approx \hom_\catl{A}(-, B)$ because $J$ is the promonoidal right unit of $M$. Similarly, $J$ is the left unit of $\otimes_M$. If the promonoidal structure on $\catl{A}$ is induced by a monoidal structure --- i.e.~if:
$$M(-, B, C) = \hom_\catl{A}(-, B \otimes_M C)$$
then this structure is preserved by the Yoneda embedding --- there is a natural isomorphism:
$$\hom_\catl{A}(-, X) \otimes_M \hom_\catl{A}(-, Y) \approx M(-, X, Y)$$
Indeed, by definition $$\hom_\catl{A}(-, X) \otimes_M \hom_\catl{A}(-, Y) = \int^{B, C \in \catl{A}} \hom_\catl{A}(B, X) \otimes \hom_\catl{A}(C, Y) \otimes M(-, B, C)$$ which via Yoneda reduction is isomorphic to $M(-, X, Y)$.
\end{example}
%
Brain Day showed more --- every monoidal structure induced via convolution is a \emph{(bi)closed} monoidal structure. The left linear exponent is defined by the end: 
$$(F \multimap_M^L G)(-) = \int_{A, C \in \catl{A}} G(A)^{F(C) \otimes M(A, -, C)}$$
and the right linear exponent by the end:
$$(F \multimap_M^R G)(-) = \int_{A, B \in \catl{A}} G(A)^{F(B) \otimes M(A, B, -)}$$
We have to show that:
$$\hom(H, F \multimap_M^L G) \approx \hom(H \otimes_M F, G)$$
Unwinding the right hand side, we get:
\begin{eqnarray*}
\hom(H \otimes_M F, G) & = & \hom(\int^{B, C \in \catl{A}} H(B) \otimes F(C) \otimes M(-, B, C), G) \\
& \approx & \int_{B, C \in \catl{A}} \hom(H(B) \otimes F(C) \otimes M(-, B, C), G) \\
& \approx & \int_{A, B, C \in \catl{A}} G(A)^{H(B) \otimes F(C) \otimes M(A, B, C)} \\
& \approx & \int_{A, B, C \in \catl{A}} (G(A)^{F(C) \otimes M(A, B, C)})^{H(B)} \\
& \approx & \int_{A, C \in \catl{A}} \hom(H, G(A)^{F(C) \otimes M(A, -, C)}) \\
& \approx & \hom(H, \int_{A, C \in \catl{A}} G(A)^{F(C) \otimes M(A, -, C)}) \\
& \approx & \hom(H, F \multimap_M^L G)
\end{eqnarray*}
and similarly for the other variable.


We show that a similar phenomenon occurs for internal categories. In \cite{hop} Brain Day and Ross Street defined a notion of convolution within a monoidal (weak) 2-category (Proposition 4). For a reason that shall become clear in a moment, we are willing to call it ``virtual convolution''. Here is their definition. Let:
$$\langle A, \delta \colon A \rightarrow A \otimes A, \epsilon \colon A \rightarrow I \rangle$$
be a weak comonoid, and:
$$\langle B, \mu \colon B \otimes B \rightarrow B, \eta \colon I \rightarrow B \rangle$$
be a weak monoid in a monoidal (weak) 2-category with tensor $\otimes$ and unit $I$, then:
$$\langle \hom(A, B), \star, i \rangle$$
is a monoidal category by:
\begin{eqnarray*}
f \star g &=& \mu \circ (f \otimes g) \circ \delta\\
i &=& \eta \circ \epsilon 
\end{eqnarray*}
So the ``virtual convolution'' structure exists ``virtually'' --- on $\hom$-categories. If a monoidal 2-category admits all right Kan liftings, then the induced monoidal category $\langle \hom(I, B), \star, i \rangle$ for trivial comonoid on $I$ is monoidal (bi)closed by:
$$f \overset{L}\multimap h = \mathit{Rift}_{\mu \circ (f \otimes \mathit{id})}(h)$$

$$f \overset{R}\multimap h = \mathit{Rift}_{\mu \circ (\mathit{id} \otimes f)}(h)$$
where $\mathit{Rift}_{\mu \circ (f \otimes \mathit{id})}(h)$ is the right Kan lifting of $h$ along $\mu \circ (f \otimes \mathit{id})$ and $\mathit{Rift}_{\mu \circ (\mathit{id} \otimes f)}(h)$ is the right Kan lifting of $h$ along $\mu \circ (\mathit{id} \otimes f)$.

Taking for the monoidal 2-category the category of distributors, we obtain the well-known formula for convolution. However, in the general setting, such induced structure is far weaker than one would wish to have --- for example in the category of distributors enriched over a monoidal category $\mathbb{C}$ the induced convolution instead of giving a monoidal structure on the category of enriched presheaves:
$$\mathbb{C}^{B^{op}}$$
merely gives a monoidal structure on the underlying ($\catw{Set}$-enriched) category\footnote{There is a work-around for this issue in the context of enriched categories, as suggested in the \cite{hop}, but the general weakness of ``virtual convolution'' is obvious.}:
$$\hom(I, \mathbb{C}^{B^{op}})$$
The solution is to find a way to ``materialize'' the ``virtual convolution''. Here is a materialization for internal categories.
\begin{theorem}[Internal convolution]\label{t:internal:convolution}
Let $\catl{C}$ be a locally cartesian closed category with finite colimits. For every $\catl{C}$-internal distributor $\dist{\mu}{A \times A}{A}$ there is a canonical (bi)closed magma on $\mathit{fam}(\mathbb{C})^{\mathit{fam}(A)^{op}}$. Furthermore, if $\dist{\mu}{A \times A}{A}$ together with $\dist{\eta}{1}{A}$ is a weak (symmetric) monoid, then the induced magma is weak (symmetric) monoidal.
\end{theorem} 
\begin{proof}
We shall present a proof for a promonoidal structure on $A$. The case of (bi)closed magma is analogical.
 
Since $\catl{C}$ is locally cartesian closed, every existing colimit in $\catl{C}$ is stable under pullbacks. In particular, coequalisers are stable under pullbacks, and we may form the (weak) 2-category of $\catl{C}$-internal distributors with compositions defined in the usual tensor-like manner (Section 3 of \cite{variations}, Section 3 of \cite{cohomology}). Moreover, local cartesian closedness allows us to ``transpose'' compositions (where coequalisers turn into equalisers, and pullbacks turn into local exponents) which makes the category of distributors admit all right Kan liftings\footnote{See \cite{variations} (Section 4)}.
We have to show that given a promonoidal structure $$\langle A, \mu \colon A \times A \nrightarrow A, \eta \colon 1 \nrightarrow A \rangle$$
there is a corresponding monoidal (bi)closed structure on:
$$\mathit{fam}(\mathbb{C})^{\mathit{fam}(A)^{op}}$$
i.e.~each fibre of $\mathit{fam}(\mathbb{C})^{\mathit{fam}(A)^{op}}$ is a monoidal closed category and reindexing functors preserve these monoidal structures. For $K \in \mathbb{C}$ interpreted as a discrete $\mathbb{C}$-internal category, there are isomorphisms:
\begin{eqnarray*}
\mathit{fam}(\mathbb{C})^{\mathit{fam}(A)^{op}}(K) & \approx & \hom(\hom(-, K), \mathit{fam}(\mathbb{C})^{\mathit{fam}(A)^{op}})\\
& \approx & \hom(1, \mathit{fam}(\mathbb{C})^{\hom(-, K) \times \mathit{fam}(A)^{op}})\\
&=& \hom_{\catw{dist}(\catl{C})}(1, K \times A)
\end{eqnarray*}
where the first isomorphism is the fibred Yoneda lemma, and the second is induced by cartesian closedness of $\catw{Cat}^{\op{\catl{C}}}$ and the fact that $K = \op{K}$ for discrete internal category $K$.

Since $K$ has a trivial promonoidal structure:
$$\tuple{K, K \times K \overset{\Delta^*}\nrightarrow K, 1 \overset{!^*}\nrightarrow K}$$
we obtain a ``product'' promonoidal structure on $K \times A$:
\begin{eqnarray*}
K \times A \times K \times A &\overset{\Delta^* \times \mu}\nrightarrow& K \times A \\
1 &\overset{\langle !^*, \eta \rangle}\nrightarrow& K \times A
\end{eqnarray*}
Explaining the above notion in more details --- observe that because $\mathbb{C}$ is cartesian, every object $K \in \mathbb{C}$ carries a unique comonoid structure:
\begin{eqnarray*}
K &\overset{\Delta}\rightarrow& K \times K \\
K &\overset{!}\rightarrow& 1
\end{eqnarray*}
which has a promonoidal right adjoint structure $\langle \Delta^*, !^* \rangle$ in the category of internal distributors. The product of the above two promonoidal structures is given by the usual cartesian product of internal categories  (note, it is not a product in the category of internal distributors) followed by the internal product functor $\mathit{fam}(\mathbb{C}) \times \mathit{fam}(\mathbb{C}) \overset{\mathit{prod}}\rightarrow \mathit{fam}(\mathbb{C})$.

Then, by ``virtual convolution'' there is a monoidal (bi)closed structure on $\hom_{\catw{dist}(\catl{C})}(1, K \times A)$. Therefore each fibre $\mathit{fam}(\mathbb{C})^{\mathit{fam}(A)^{op}}(K)$ is a monoidal (bi)closed category. Since pullback functors preserve equalisers, and in a locally cartesian category pullback functors preserve local exponents, they also preserve the convolution structure.
\end{proof}
Let us work out the concept of internal Day convolution in case $\mathbb{C} = \mathbf{Set}$, and see that it agrees with the usual formula for convolution.
\begin{example}[$\catw{Set}$-internal convolution]
The split family fibration (or more accurately, the indexed functor corresponding to the family fibration) for a (locally) small category $A$:
$$\mathit{fam}(A) \colon \mathbf{Set}^{op} \rightarrow \mathbf{Cat}$$
is defined as follows:
\begin{eqnarray*}
\mathit{fam}(A)(K \in \mathbf{Set}) &=& A^K \\ 
\mathit{fam}(A)(K \overset{f}\rightarrow L) &=& A^L \overset{(-) \circ f}\rightarrow A^K
\end{eqnarray*}
where $K, L$ are sets and $K \overset{f}\rightarrow L$ is a function between sets. One may think of category $A^K$ as of the category of $K$-indexed tuples of objects and morphisms from $A$. Given any monoidal structure on a small category $$\langle A, \otimes \colon A \times A \rightarrow A, I \colon 1 \rightarrow A \rangle$$
the usual notion of convolution induces a monoidal structure on $\mathbf{Set}^{A^{op}}$:
$$\langle F, G \rangle \mapsto F \otimes G =  \int^{B, C \in A} F(B) \times G(C) \times \hom(-, B \otimes C)$$
The split fibration:
$$\mathit{fam}(\mathbf{Set})^{\mathit{fam}(A)^{op}} \colon \mathbf{Set}^{op} \rightarrow \mathbf{Cat}$$
may be characterised as follows:
\begin{eqnarray*}
\mathit{fam}(\mathbf{Set})^{\mathit{fam}(A)^{op}}(K \in \mathbf{Set}) &=& \mathbf{Set}^{A^{op} \times K} \\ 
\mathit{fam}(\mathbf{Set})^{\mathit{fam}(A)^{op}}(K \overset{f}\rightarrow L) &=& \mathbf{Set}^{A^{op} \times L} \overset{(-) \circ (\mathit{id} \times f)}\rightarrow \mathbf{Set}^{A^{op} \times K}
\end{eqnarray*}
Since $\mathbf{Set}^{A^{op} \times K} \approx (\mathbf{Set}^{A^{op}})^K$ we may think of $\mathbf{Set}^{A^{op} \times K}$ as of $K$-indexed tuples of functors ${A^{op} \rightarrow \mathbf{Set}}$. In fact:
$$\mathit{fam}(\mathbf{Set})^{\mathit{fam}(A)^{op}} \approx \mathit{fam}(\mathbf{Set}^{A^{op}})$$
It is natural then to extend the monoidal structure induced on $\mathbf{Set}^{A^{op}}$ pointwise to $(\mathbf{Set}^{A^{op}})^K$:
$$(F \otimes G)(k) = \int^{B, C \in A} F(k)(B) \times G(k)(C) \times \hom(-, B \otimes C)$$
where $k \in K$.
On the other hand, using the internal formula for convolution, we get (up to a permutation of arguments):
\begin{center}
\begin{tabular}{c}
$\int^{B, C \in A, \beta, \gamma \in K} F(B, \beta) \times G(C, \gamma) \times \hom(\Delta(k), \langle \beta, \gamma \rangle) \times \hom(-, B \otimes C)$ \\ \hline\hline
$\int^{B, C \in A, \beta, \gamma \in K} F(B, \beta) \times G(C, \gamma) \times \hom(k, \beta) \times \hom(k, \gamma) \times \hom(-, B \otimes C)$ \\ \hline\hline
$\int^{B, C \in A} F(B, k) \times G(C, k) \times \hom(-, B \otimes C)$ \\
\end{tabular}
\end{center}
where the first equivalence is the definition of a diagonal $\Delta$, and the second one is by ``Yoneda reduction'' applied twice.
\end{example}
Note that the local cartesian closedness of the ambient category $\catl{C}$ was crucial for the proof. There is always the trivial (cartesian) monoidal structure on the terminal category $1$ internal to $\catl{C}$, but if $\catl{C}$ is not locally cartesian closed than its fundamental fibration $\mathit{fam}(\catl{C}) \approx \mathit{fam}(\catl{C})^1$ is not a cartesian closed fibration. 


There is another, more abstract, road to Day convolution for internal categories. Recall from \cite{log1} (Section 4, Definition 18) that if $\mor{F}{\catl{C}}{\catl{W}}$ is a functor from a 1-category $\catl{C}$ to a 2-category $\catl{W}$, then the $F$-externalisation $\word{fam}_F(A)$ of an object $A \in \catl{W}$ is defined to be the functor:
$$\mor{\hom_\catl{W}(F(-), A)}{\catl{C}^{op}}{\catw{Cat}}$$
For example, in Theorem \ref{e:cat:internal:triangle}, $\word{fam}(A)$ is an $F$-externalisation of an object (i.e.~internal category) $A \in \catw{cat}(\catl{C})$. However, the 2-power $P(A)$ may be itself defined as an ``externalisation'' --- namely, the $J \circ F$ ``externalisation'' of $A$. By fibred Yoneda lemma $P(A) \approx \word{fam}_{J \circ F}(A)$ iff there is a natural in $X \in \catl{C}$ isomorphism:
$$\hom_{\cat^{\op{\catl{C}}}}(\hom_\catl{C}(-, X), P(A)) \approx \hom_{\cat^{\op{\catl{C}}}}(\hom_\catl{C}(-, X), \word{fam}_{J \circ F}(A))$$
The left hand side by definition is isomorphic to $\hom_{\catw{dist}(\catl{C})}(J(F(X)), A)$, and by Fibred Yoneda lemma, the right hand side is isomorphic to $\word{fam}_{J \circ F}(A)(X)$, which by definition equals $\hom_{\catw{dist}(\catl{C})}(J(F(X)), A)$.
%

Therefore, the Yoneda bitriangle for internal powers may be redrawn as follows:

$$\bfig
\node c(0,0)[\catl{C}]
\node pc(0,400)[\cat^{\op{\catl{C}}}]
\node cat(800,0)[\cat(\catl{C})]
\node pcat(800,400)[\cat^{\op{\cat(\catl{C})}}]
\node dist(1600,0)[\catw{dist}(\catl{C})]
\node pdist(1600,400)[\cat^{\op{\catw{dist}(\catl{C})}}]

\arrow/->/[c`cat;F]
\arrow/->/[cat`dist;J]
\arrow|m|/-->/[c`pc;y_{\catl{C}}]
\arrow|m|/->/[cat`pcat;y_{\cat(\catl{C})}]
\arrow|m|/->/[dist`pdist;y_{\catw{dist}(\catl{C})}]
\arrow[pcat`pc; (-) \circ \op{F}]
\arrow[pdist`pcat; (-) \circ \op{J}]

\place(1250, 200)[\twoar(1, 0)\scriptstyle{\;J_1}]
\place(400, 200)[=]

\efig$$
where $\mor{F}{\catl{C}}{\cat(\catl{C})}$ is a strong (cartesian) monoidal functor, $\mor{J}{\cat(\catl{C})}{\catw{dist}(\catl{C})}$ is strong monoidal by the definition of tensor product on $\catw{dist}(\catl{C})$. Both $\op{F} \circ (-)$ and $\op{J} \circ (-)$ are lax monoidal, because $F$ and $G$ are strong monoidal. Moreover, natural transformation:
$$\mor{J_1}{\hom_{\catw{dist}(\catl{C})}(J(\ph{1}), J(\ph{2}))}{\hom_{\cat(\catl{C})}(\ph{1}, \ph{2})}$$
induced by the ``arrows-part'' of monoidal functor $\mor{J}{\cat(\catl{C})}{\catw{dist}(\catl{C})}$ is monoidal. 

Therefore we shall introduce the following concept.
\begin{definition}[Monoidal Yoneda (bi)triangle]\label{d:monoidal:yoneda:triangle}
A Yoneda (bi)triangle $\yon{\eta}{y}{f}{g}$ is monoidal if $f$ and $g$ are lax monoidal morphisms between (weak) monoidal objects, and $\eta$ is a monoidal 2-morphism.
\end{definition}

There are various possibilities to define universes that induce free semantics. Here is the weakest one.

\begin{definition}[Power semantics universe]\label{d:power:semantics:universe}
Let $\yon{\eta}{Y}{F}{G}$ be a monoidal Yoneda bitriangle $\mor{Y}{\catl{A}}{\overline{\catl{A}}}$, $\mor{F}{\catl{A}}{\catl{B}}$, $\mor{G}{\catl{B}}{\overline{\catl{A}}}$, where $\catl{A}$ admits a notion of discreteness, and $\overline{\catl{A}}$ has finite coproducts and admits a notion of discreteness with op-lax monoidal free functors. We call bitriangle $\yon{\eta}{Y}{F}{G}$ the power universe, if magmas mapped by $G$ are (bi)closed, $Y$ preserves discrete objects, and for every $V \in \disc(\catl{A})$ the functor:
$$F_V(X) \mapsto Y(V) \sqcup (X \otimes X) \sqcup (|X| \otimes X) \sqcup (|X| \otimes X)$$
has an initial algebra $\word{Lambek}_V$. 
\end{definition}
If $\yon{\eta}{Y}{F}{G}$ is a power semantics universe, then for every magma $\mor{R}{M \otimes M}{M}$ and every $\mor{\models}{V}{M}$ in $\catl{B}$, the free semantics of $V$ by $R$ is defined to be the the unique ``semantic'' morphism $\word{Lambek}_V \rightarrow P(M)$ from the initial algebra $\word{Lambek}_V$ to the algebra:
$$Y(V) \sqcup (P(M) \otimes P(M)) \sqcup (|P(M)| \otimes P(M)) \sqcup (|P(M)| \otimes P(M)) \to^{\models, \otimes_R, \overset{L}\multimap_R, \overset{R}\multimap_R} P(M)$$
where $\otimes_R, \overset{L}\multimap_R, \overset{R}\multimap_R$ is the internally (bi)closed magma on $P(M)$ induced by $R$. Similarly, we may introduce power semantics universe for (weak) monoids. Moreover, in many cases (universes induced by categories enriched over a complete and cocomplete symmetric monoidal category, and in universes induced by categories internal to a finitely cocomplete locally cartesian closed category) power objects $P(M)$ are internally cocomplete, thus, in particular, have internal coproducts. This observation makes it possible to extend the above semantics by propositional disjunctions and ``false'' value.
\begin{example}[Kripke semantics]\label{e:kripke:module}
A Kripke structure is a triple \newline $\tuple{S, {\leq} \subset {S \times S}, {\Vdash} \subseteq {S \times |\word{Prop}_V|}}$, where $\leq$ is a partial order on $S$, $\word{Prop}_V$ is the propositional syntax on a set of variables $V$, and $\Vdash$ is a ``forcing'' relation satisfying:
\begin{itemize}
\item (compatibility on variables) if $A \in V$ and $p,q \in S$ such that $p \leq q$ then ${p \Vdash A} \Rightarrow {q \Vdash A}$ 
\item (extensional true) $p \Vdash \top$ always holds
\item (extensional false) $p \Vdash \bot$ never holds
\item (extensional and) $p \Vdash \phi \wedge \psi$ iff $p \Vdash \phi$ and $p \Vdash \psi$
\item (extensional or) $p \Vdash \phi \vee \psi$ iff $p \Vdash \phi$ or $p \Vdash \psi$
\item (extensional implication) $p \Vdash \phi \Rightarrow \psi$ iff for all $q \in S$ such that $p \leq q$ we have: $q \Vdash \phi$ implies $q \Vdash \psi$
\end{itemize}
The compatibility condition on variables implies compatibility condition on all formulae, so every Kripke structure gives rise to logical system ${\Vdash \colon {\langle S, \geq \rangle}^{op} \times \word{Prop}_V \rightarrow 2}$, where $\langle S, \geq \rangle^{op} = \langle S, \leq \rangle$ is a degenerated category, and $\word{Prop}_V$ is the category induced by the logical consequence of $\Vdash$.

Kripke structures may be rediscovered as power semantics for trivial comonoidal structure in the power semantics universe of $2$-enriched categories. A poset ${\leq} \subseteq {S \times S}$ is exactly a $2$-enriched category $S$. Moreover, $S$ has the trivial comonoidal structure $\mor{\Delta}{S}{S \times S}$, which induces a promonoidal structure $\dist{\Delta^*}{S\times S}{S}$.

Given a ``forcing'' relation on variables ${\Vdash_V} \subseteq {S \times V}$ that satisfies compatibility condition (i.e.~is a 2-enriched distributor $\dist{{\Vdash_V}}{V}{\op{S}}$), there is the semantic homomorphism $\word{Lambek}_V \rightarrow 2^S$ induced by initiality of $\word{Lambek}_V$ and algebraic structure $\tuple{{\Vdash_V}, {\times}, {\overset{L}\Rightarrow}, {\overset{R}\Rightarrow}}$, where ${\times} = 2^{\Delta^*}$ is the usual cartesian product. Observe that since $\Delta^*$ is symmetric, both exponents ${\overset{L}\Rightarrow}$ and ${\overset{R}\Rightarrow}$ are essentially the same, and we may drop one of them from our signature. Furthermore, because $\Delta^*$ has also a unit, and $2$-enriched presheaves are cocomplete, one may extend the signature functor by additional operations representing true/false objects and disjunctions:
$$L_V(X) \mapsto Y(V) \sqcup 1 \sqcup 1 \sqcup (X \times X) \sqcup (|X| \times X)$$
The initial algebra for $L_V$ is the discrete propositional category $|\mathit{Prop}_V|$ and the Kripke semantics ${\Vdash} \subseteq {S \times |\mathit{Prop}_V|}$ is obtained as the transposition of the unique homomorphism to the algebra $\tuple{\Vdash_V, 1, 0, \times, \sqcup, \Rightarrow}$.
\end{example}
Let us recall the following example from \cite{log1}.
\begin{example}[Logical consequence]\label{e:logical:consequence:codensity}
Let $\onecat(2)$ be the 2-category of categories enriched in a 2-valued Boolean algebra $2 = \{0 \rightarrow 1\}$. A $2$-enriched category is tantamount to a partially ordered set (poset), and a $2$-enriched functor is essentially a monotonic function between posets. Let us consider a relation:
$${\models} \subseteq {\word{Mod} \times \word{Sen}}$$
thought of as a satisfaction relation between a set of models $\word{Mod}$ and a set of sentences $\word{Sen}$. By transposition, relation $\models$ yields the ``theory'' function $\mor{\word{th}}{\word{Mod}}{2^\word{Sen}}$, where $2^\word{Sen}$ is the poset of functions $\word{Sen} \rightarrow 2$, or equivalently the poset of subsets of $\word{Sen}$ .

Since ``power'' posets $2^\word{Sen}$ are internally complete in the 2-category $\onecat(2)$, the stable density product of $\mor{\word{th}}{\word{Mod}}{2^\word{Sen}}$ exists:
$$\bfig
\node x(0, 1000)[\word{Mod}]
\node c1(600, 1000)[2^\word{Sen}]
\node c2(0, 700)[2^\word{Sen}]

\arrow/->/[x`c1;\word{th}]
\arrow/->/[x`c2;\word{th}]
\arrow|r|/->/[c2`c1;T_\word{th}]
\place(230, 860)[\twoar(1, 1)]
\efig$$
and is given by the 2-enriched end:
$$T_\word{th}(\Gamma)(\psi) = \int_{M\in \mathit{Mod}} \word{th}(M)(\psi)^{\hom(\Gamma, \word{th}(M)(-))}$$
where $\psi \in \word{Sen}$ is a sentence, and $\Gamma \in 2^\word{Sen}$ is a set of sentences. We are interested in values of $T_\word{th}$ on representable functors (i.e.~single sentences) $\hom_\word{Sen}(-, \phi)$:
\begin{eqnarray*}
T_\word{th}(\hom_\word{Sen}(-, \phi))(\psi) &=& \int_{M\in \mathit{Mod}} \word{th}(M)(\psi)^{\hom(\hom_\word{Sen}(-, \phi), \word{th}(M)(-))}\\
&\approx& \int_{M\in \mathit{Mod}} \word{th}(M)(\psi)^{\word{th}(M)(\phi)}
\end{eqnarray*}
where the isomorphism follows from the Yoneda reduction. Observe that the exponent $\word{th}(M)(\psi)^{\word{th}(M)(\phi)}$ in a $2$-enriched world may be expressed by the implication ``${\word{th}(M)(\phi)} \Rightarrow \word{th}(M)(\psi)$'', or just ``${M \models \phi} \Rightarrow {M \models \psi}$'', where every component of the implication is interpreted as a logical value in the $2$-valued Boolean algebra. Furthermore, ends turn into universal quantifiers, when we move to $2$-enriched world. So, the end $\int_{M \in \mathit{Mod}} \word{th}(M)(\psi)^{\word{th}(M)(\phi)}$ is equivalent to the meta formula ``$\forall_{M\in \word{Mod}} \left({M \models \phi} \Rightarrow {M \models \psi}\right)$'', which is just the definition of logical consequence:
$$ \phi \models_\word{Sen} \psi \;\;\;\mathit{iff}\;\;\; \forall_{M\in \word{Mod}} \left({M \models \phi} \Rightarrow {M \models \psi}\right)$$
The general case, where $\Gamma$ is not necessarily representable, is similar:
$$T_\word{th}(\Gamma)(\psi)  \;\;\;\mathit{iff}\;\;\; \forall_{M\in \word{Mod}} \left(\left(\forall_{\phi \in \Gamma} M \models \phi\right) \Rightarrow {M \models \psi}\right)$$
Therefore, the density product of a satisfaction relation reassembles the semantic consequence relation. 
\end{example}
In this example the satisfaction relation ${\models} \subseteq {\word{Mod} \times \word{Sen}}$ induces semantic consequence relation ${\models_\word{Sen}} \subseteq {\word{Sen} \times \word{Sen}}$ via the density product. We have also seen in \cite{log1}, that density products are always equipped with a monad structure. In fact ${{\models_\word{Sen}} \subseteq {\word{Sen} \times \word{Sen}}}$ thought of as a $2$-enriched distributor acquires the monad structure from the density product. Because the 2-category of $2$-enriched distributors is cocomplete, this monad has a resolution as a Kleisli object $\word{Sen}_K$. In more details, a Kleisli object in the 2-category of $2$-enriched distributors may be described by generalised Grothendieck construction \cite{daw} --- objects in $\word{Sen}_K$ are the same as in $\word{Sen}$, whereas morphisms in $\word{Sen}_K$ are defined by:
$$\hom_{\word{Sen}_K}(\phi, \psi) = \phi \models_\word{Sen} \psi$$
Identities and compositions are induced by monad's unit and multiplication respectively. Then by definition of density product the relation ${{\models} \subseteq {\word{Mod} \times \word{Sen}}}$ extends to the relation:
$${\models} \subseteq {\word{Mod} \times \word{Sen}_K}$$
In an essentially the same manner one may extend the forcing relation  ${{\Vdash} \subseteq {S \times |\mathit{Prop}_V|}}$ of the above Example \ref{e:kripke:module} to the relation:
$${\Vdash} \subseteq {S \times \mathit{Prop}_V}$$
where $\mathit{Prop}_V$ is the Kleisli resolution for the density product on the forcing relation. 

The next example generalises semantics in Kripke structures.
\begin{example}[Ternary frame]
A ternary frame \cite{ternarysemantics} is a pair $\tuple{X, R}$, where $X$ is a set, and $\mor{R}{X \times X \times X}{2}$ is a ternary relation on $X$. Ternary frames were proposed as generalisations of Kripke structures to model substructural logics. Let $\Sigma_\word{Lambek}$ be the signature consisting of three binary symbols $\otimes$, $\overset{L}\multimap$ and  $\overset{R}\multimap$. The semantics for Lambek syntax in ternary frame $\tuple{X, R}$ is a relation ${\Vdash} \subseteq {X \times \word{Lambek}_\mathit{Var}}$ satisfying:
\begin{itemize}
\item $x \Vdash \phi \otimes \psi$ iff $\underset{y, z \in X}\exists\; y \Vdash \phi \wedge z \Vdash \psi \wedge R(x, y, z)$
\item $y \Vdash \phi \overset{L}\multimap \psi$ iff $\underset{x, z \in X}\forall\; z \Vdash \phi \wedge R(x, y, z) \Rightarrow x \Vdash \psi$
\item $z \Vdash \phi \overset{R}\multimap \psi$ iff $\underset{x, y \in X}\forall\; y \Vdash \phi \wedge R(x, y, z) \Rightarrow x \Vdash \psi$
\end{itemize}
Connectives are defined according to the nonassociative Lambek calculus defined on $2^X$ via the convolution of $R$. 
\end{example}

\section{Conclusions}
\label{s:conclusions}
In the paper we defined a general 2-categorical setting for extensional calculi and showed how intensional and extensional calculi can be combined into logical systems. We provided a notion of a generalised adjunction, which we call a Yoneda triangle, and showed that many concepts in category theory may be characterised as (higher) monoidal Yoneda triangles. We showed that the natural setting for convolution is a monoidal Yoneda bitriangle, and proved Day convolution theorem for internal categories. Such  monoidal Yoneda bitriangles provide semantic universes, where objects get their semantics (almost) for free --- this includes the usual semantics for propositional calculi, Kripke semantics for intuitionistic calculi and ternary frame semantics for substructural calculi including Lambek's lambda calculi, relevance and linear logics.

\end{document}